\documentclass[a4paper]{article}

\usepackage{relsize}
\usepackage{hyperref}
\hypersetup{colorlinks, linkcolor=blue, urlcolor=red, filecolor=green, citecolor=blue}

\usepackage{fancyhdr}
\medskipamount

\usepackage{amsfonts}
\usepackage{amsmath}
\usepackage{amsthm}
\usepackage{amssymb}
\usepackage{amscd}
\usepackage{enumerate}
\usepackage{esint}
\usepackage{mathtools}

\theoremstyle{plain}
\newtheorem{theorem}{Theorem}[section]

\newtheorem{lemma}[theorem]{Lemma}
\newtheorem{proposition}[theorem]{Proposition}

\theoremstyle{definition}
\newtheorem{assumption}[theorem]{Assumption}
\newtheorem{definition}[theorem]{Definition}

\theoremstyle{remark}

\newtheorem*{remarks}{Remarks}

\newcommand\R{\mathbb R}

\newcommand\Z{\mathbb Z}

\renewcommand{\a}{\alpha}
\renewcommand{\b}{\beta}
\newcommand{\e}{\varepsilon}
\newcommand{\eh}{\varepsilon(h)}

\newcommand{\SO}[1]{\operatorname{SO}(#1)}
\newcommand\id{I}
\newcommand\sym{\operatorname{sym}}

\newcommand{\WW}{{W^{2,2}_{\delta}}}

\newcommand\dist{\operatorname{dist}}

\newcommand\wto{\rightharpoonup}

\newcommand\wtto{\xrightharpoonup{2,\gamma}}
\newcommand\otto{\xrightharpoonup{osc,\gamma}}
\newcommand\stto{\xrightarrow{2,\gamma}}

\DeclareMathOperator{\spt}{spt}
\DeclareMathOperator{\secf}{II}

\renewcommand{\t}{\widetilde}

\title{Derivation of a homogenized nonlinear plate theory from 3d elasticity}
\date{}

\author{Peter Hornung \and Stefan Neukamm \and  Igor Vel\v{c}i\'{c}}

\begin{document}

\maketitle

\begin{abstract}
We derive, via simultaneous homogenization and dimension reduction,
the $\Gamma$-limit for thin elastic plates whose energy density
oscillates on a scale that is either comparable to, or much smaller than,
the film thickness. We consider the energy scaling that corresponds
to Kirchhoff's nonlinear bending theory of plates.

 \vspace{10pt}

\end{abstract}

\section{Introduction}

Kirchhoff's nonlinear plate theory associates with a
deformation $u : S\to\R^3$ of a two-dimensional stress-free reference configuration $S\subset\R^2$
the bending energy
\begin{equation}\label{Neu:eq:1}
  \int_S Q_2(\secf(x'))\,dx',
\end{equation}
where $Q_2$ is the quadratic form from linear
elasticity, and $\secf$ denotes the second
fundamental form associated with $u$. The key condition on the admissible deformations $u$ is that
they must satisfy the isometry constraint
\begin{equation}
  \label{eq:isometry}
  \partial_\alpha
u\cdot\partial_\beta u=\delta_{\alpha\beta},\qquad\alpha,\beta\in\{1,2\}
\end{equation}
where $\delta_{\alpha\beta}$ denotes the Kronecker delta. Physically,
\eqref{Neu:eq:1} describes the elastic energy stored in a deformed plate that
can undergo large deformations but not shearing or stretching.

\newcommand{\super}[1]{^{(#1)}}
\def\uah{\boldsymbol u^{(h)}}
\def\vah{\boldsymbol v^{(h)}}
In \cite{FJM1} Kirchhoff's nonlinear plate
theory was rigorously derived as a zero-thickness $\Gamma$-limit from 3d nonlinear
elasticity. In this article we combine their result with homogenization. We consider
a plate with thickness $h\ll 1$ made of a composite material that
periodically oscillates with period $\e\ll 1$ in in-plane directions. We shall derive a
homogenized plate model via simultaneous homogenization and dimension
reduction in the limit $(h,\e)\to 0$  when the material period $\e$
and the thickness $h$ are either comparable or behave as $\e\ll h$, see
Theorem~\ref{T:main} below. The derived model is sensitive to the relative scaling of $h$ and $\e$. Our result generalizes recent results from \cite{Neukamm-12}
where the one-dimensional case is studied.
Regarding plates, related results have been obtained for
different energy scalings: In \cite{Braides-Fonseca-Francfort-00,
  Babadijan-Baia-06} the membrane regime has been
considered. Recently,

the energy scaling corresponding to the \mbox{von-K\'arm\'an} plate model was studied in \cite{NeuVel}.

This article is organized as follows. Section~\ref{S:framework} introduces
the general framework and discusses the main results.
In Section~\ref{S:two-scale} we recall the notion of two-scale
convergence and characterize the two-scale limit of nonlinear strains.
In Section~\ref{S:proof} we prove our main result.

\section{General framework and main result}\label{S:framework}

From now on, $S\subset\R^2$ denotes a bounded Lipschitz domain whose boundary is piecewise $C^1$.
The piecewise $C^1$-condition is necessary only for the proof of the upper bound and can be slightly
relaxed, cf. \cite{Hornung-arma2}.
\\
For $h > 0$ and $I:=(-\frac{1}{2},\frac{1}{2})$, we denote by
$\Omega_h:=S\times hI$ the reference configuration of the thin plate of thickness $h$.
The elastic energy per unit volume associated with a deformation $v^h:\Omega_h\to\R^3$
is given by
\begin{equation}\label{Neu:eq:0}
  \frac{1}{h}\int_{\Omega_h}W(\frac{z'}{\e},\nabla v^h(z))\,dz.
\end{equation}
Here and below $z'=(z_1,z_2)$ stands for the in-plane coordinates of a generic
element $z=(z_1,z_2,z_3)\in\Omega_h$ and $W$ is a energy
density that models the elastic properties of a periodic composite.
\begin{assumption}\label{assumption}
  We assume that
  \begin{equation*}
    W:\R^2\times\R^{3\times
      3}\to [0,\infty],\qquad (y,F)\mapsto W(y,F)
  \end{equation*}
  is measurable and $[0,1)^2$-periodic in $y$ for all $F$. Furthermore, we assume that for
  almost every $y\in\R^2$, the map $\R^{3\times 3}\ni F\mapsto W(y,F)\in
  [0,\infty]$ is continuous and satisfies the following properties:
  \begin{align}
    \tag{FI}\label{ass:frame-indifference}
    &\text{(frame indifference)}\\
    &\notag\qquad W(y, RF)=W(y, F)\quad\text{ for
      all $F\in\R^{3\times 3}$, $R\in\SO 3$;}\\[0.2cm]
    \tag{ND}\label{ass:non-degenerate}
    &\text{(non degeneracy)}\\
    &\notag\qquad W(y,  F)\geq c_1\dist^2( F,\SO 3)\quad\text{ for all
      $ F\in\R^{3\times 3}$;}\\
    &\notag\qquad W(y, F)\leq c_2\dist^2( F,\SO 3)\quad\text{ for all
      $ F\in\R^{3\times 3}$ with $\dist^2( F,\SO 3)\leq\rho$;}\\[0.2cm]
    \tag{QE}\label{ass:expansion}
    &\text{(quadratic expansion at identity)}\\
    &\notag\qquad
    \lim_{G\to 0}\frac{W(y, \id+ G) - Q(y, G)}{|G|^2} = 0\\
    &\notag\qquad\text{for some quadratic form $Q(y,\cdot)$ on $\R^{3\times 3}$.}
  \end{align}
  Here $c_1,c_2$ and $\rho$ are positive constants which are
  fixed from now on.
\end{assumption}
We define $\Omega := S\times I$. As in \cite{FJM1} we rescale the out-of-plane coordinate: for
$x = (x',x_3)\in\Omega$ consider the scaled deformation
$u^h(x',x_3):=v^h(x',hx_3)$. Then \eqref{Neu:eq:0} equals
\begin{equation}
  \label{eq:3'}
  \mathcal E^{h,\e}(u^h):=\int_{\Omega}W(\frac{x'}{\e},\nabla_h u^h(x))\,dx,
\end{equation}
where $\nabla_h u^h:=\big(\,\nabla' u^h,\,\frac{1}{h}\partial_3
u^h\,\big)$ denotes the scaled gradient, and
$\nabla'u^h:=\big(\,\partial_1u^h,\,\partial_2 u^h\,\big)$ denotes the
gradient in the plane.

We recall some known results on dimension reduction in the homogeneous case when  $W(y,F)=W(F)$. As
explained in \cite{FJM2} a hierarchy of
plate models can be derived from $\mathcal E^h:=\mathcal E^{h,1}$ in the zero-thickness limit $h\to 0$.
The different limiting models are distinguished by the scaling of the
elastic energy relative to the thickness. In \cite{LeDret-Raoult-93}
it is shown that the scaling  $\mathcal E^h\sim 1$ leads to a membrane model, which is a fully nonlinear plate
model for plates without resistance to compression. In the regime $\mathcal E^h\sim h^4$
finite energy deformations converge to rigid deformations and, as
shown in \cite{FJM2}, $h^{-4}\mathcal E^h$
 converges to a plate model of ``von-K\'arm\'an''-type.
n

In this article we study the \textit{bending regime} $\mathcal
E^h\sim h^2$, which, as shown in \cite{FJM1}, leads to Kirchhoff's nonlinear plate
model: as $h\to 0$ the energy $h^{-2}\mathcal E^h$
$\Gamma$-converges to the functional
\eqref{Neu:eq:1}, with $Q_2 : \R^{2\times 2}\to\R$
given by the relaxation formula
\begin{equation*}
  Q_2(A)=\min_{d\in\R^3}Q\left(\sum_{\alpha,\beta=1}^2A_{\alpha\beta}(e_\alpha\otimes
  e_\beta)+d\otimes e_3\right);
\end{equation*}
here, $Q$ denotes the quadratic form from \eqref{ass:expansion}.

We will see that in the non-homogeneous case the effective
quadratic form $Q_2$ is determined by a relaxation formula that
is more complicated and requires the solution of a corrector problem.
In particular, our analysis shows that in-plane oscillations of the deformation couple with the behavior
in the out-of-plane direction. As a consequence the effective behavior
will depend on the relative scaling between the thickness $h$ and the
material period $\e$. To make this precise we assume that $\e$ and $h$
are coupled as follows:

\begin{assumption}\label{assumption:gamma}
  Let $\gamma\in(0,\infty]$ denote a constant which is fixed
  throughout this article. We assume that $\e=\eh$ is a monotone
  function from $(0,\infty)$ to $(0,\infty)$ such that $\eh\to 0$ and
  $\frac{h}{\eh}\to \gamma$ as $h\to 0$.
\end{assumption}

The effective behavior of the homogenized plate with reduced
dimension can be computed by means of a relaxation formula that we
introduce next. We need to introduce some function spaces of periodic functions.
From now on, $Y = [0,1)^2$, and we denote by $\mathcal Y$ the set $Y$ endowed
with the torus topology, so that functions on $\mathcal Y$ will be $Y$-periodic.
\\
We write $C(\mathcal Y)$, $C^k(\mathcal Y)$ and $C^\infty(\mathcal Y)$ for the Banach
spaces of $Y$-periodic functions on $\R^2$ that are continuous,
$k$-times continuously differentiable and smooth, respectively.
Moreover, $H^1(I\times\mathcal Y)$ denotes the closure of
$C^\infty(I,C^\infty(\mathcal Y))$ with respect to the norm in
$H^1(I\times Y)$ and we write $\mathring H^1(I\times\mathcal Y)$ for the subspace of
functions $f\in H^1(S\times\mathcal Y)$ with $\iint_{I\times Y}f=0$.
The definitions extend in the obvious way to vector-valued functions.

\begin{definition}[Relaxation formula]
  \label{D:relaxation-formula}
  Let $Q$ be as in Assumption~\ref{assumption}. For $x_3\in I$ and $A,B\in\R^{2\times
    2}$, define
  \begin{equation*}
    \Lambda(x_3, A,B):=\left(\sum_{\alpha,\beta=1}^2(B_{\alpha\beta}+x_3
      A_{\alpha\beta})(e_\alpha\otimes e_\beta)\right).
  \end{equation*}
  \begin{enumerate}[(a)]
  \item   For $\gamma\in(0,\infty)$ we define
    $Q_{2,\gamma}:\R^{2\times2}_{\sym}\to[0,\infty)$ by
    \begin{eqnarray*}
      Q_{2,\gamma}(A)&:=&
      \inf_{B,\phi}\iint_{I\times Y}Q\left(
        y,\;\Lambda(x_3,A,B)+(\nabla_y
        \phi\,,\,\tfrac{1}{\gamma}\partial_3\phi)\,\right)\,dy\,dx_3
    \end{eqnarray*}
    where the infimum is taken over all $B\in\R^{2\times 2}_{\sym}$
    and $\phi\in H^1(I\times\mathcal Y,\R^3)$.
  \item   For $\gamma=\infty$ we define
    $Q_{2,\infty}:\R^{2\times2}_{\sym}\to[0,\infty)$ by
    \begin{eqnarray*}
      Q_{2,\infty}(A)&:=&
      \inf_{B,\phi, d}\iint_{I\times Y}Q\left(
        y,\;\Lambda(x_3, A,B)+(\nabla_y
        \phi\,,\,d)\,\right)\,dy\,dx_3
    \end{eqnarray*}
    where the infimum is taken over all $B\in\R^{2\times 2}_{\sym}$,
    $\phi\in L^2(I,H^1(\mathcal Y,\R^3))$ and $d\in L^2(I,\R^3)$
  \end{enumerate}
  \end{definition}

Kirchhoff's plate model is defined for pure \textit{bending deformations} of $S$ into $\R^3$;
precisely:
\begin{eqnarray}
  \WW(S,\R^3) &:=& \Bigg\{\,u\in W^{2,2}(S,\R^3)\, : \, u\text{ satisfies }\eqref{eq:isometry}\mbox{ a.e. in }S\,\Bigg\}.
\end{eqnarray}
With each $u \in \WW(S)$ we associate
its normal $n :=\partial_1 u \wedge \partial_2 u$,
and we define its second fundamental form $\secf: S\to\R^{2\times 2}_{\sym}$
by defining its entries as
\begin{equation}
  \secf_{\alpha\beta} = \partial_\alpha u\cdot\partial_\beta n
= -\partial_{\alpha}\partial_{\beta} u\cdot n.
\end{equation}
We write $\secf^h$ and $n^h$ for the second fundamental form
and normal associated with some $u^h\in\WW(S,\R^3)$.
The $\Gamma$-limit is a functional of the form \eqref{Neu:eq:1}
trivially extended to $L^2(\Omega,\R^3)$ by infinity: for $\gamma\in(0,\infty]$ define $\mathcal E_\gamma:L^2(\Omega,\R^3)\to[0,\infty]$,
\begin{equation*}
  \mathcal E_\gamma(u):=
  \left\{%
    \begin{aligned}
      &\int_S Q_{2,\gamma}(\secf(x'))\,dx'&\quad&\text{if }u\in\WW(S,\R^3)\\
      &+\infty&&\text{otherwise.}
    \end{aligned}%
    \right.
\end{equation*}
We tacitly identify functions on $S$ with their trivial
extension to $\Omega=S\times I$: above $u\in\WW(S,\R^3)$ means that $u(x',x_3)=\overline u(x'):=\fint_I
u(x',z)\,dz$ for almost every $x_3\in I$, and $\overline u\in\WW(S,\R^3)$.
Our main result is the following:
\begin{theorem}
  \label{T:main}
  Suppose that Assumptions~\ref{assumption} and \ref{assumption:gamma}
  are satisfied. Then:
  \begin{enumerate}[(i)]
  \item (Lower bound). If $\{u^h\}_{h>0}$ is a sequence with
    $u^h-\fint_\Omega u^h\,dx\to
    u$ in $L^2(\Omega,\R^3)$, then
    \begin{equation*}
      \liminf\limits_{h\to 0}h^{-2}\mathcal E^{h,\eh}(u^h)\geq \mathcal E_\gamma(u).
    \end{equation*}
  \item (Upper bound). For every $u\in\WW(S,\R^3)$ there exists a
    sequence $\{u^h\}_{h>0}$ with $u^h\to u$ strongly in
    $L^2(\Omega,\R^3)$ such that
    \begin{equation*}
      \lim\limits_{h\to 0}h^{-2}\mathcal E^{h,\eh}(u^h)=\mathcal E_\gamma(u).
    \end{equation*}
  \end{enumerate}
\end{theorem}
This theorem is complemented by the following compactness result from
\cite{FJM1}, which in particular shows that
$\{\mathcal E^{h,\eh}\}_{h>0}$ is equi-coercive on $L^2(\Omega,\R^3)$.
\begin{theorem}[\mbox{\cite[Theorem~4.1]{FJM1}}]
  \label{T:FJM-compactness}
  Suppose a sequence $u^h\in H^1(\Omega,\R^3)$ has finite
  bending energy, that is
  \begin{equation*}
    \limsup\limits_{h\to 0}\frac{1}{h^2}\int_{\Omega}\dist^2(\nabla_h
    u^h(x),\SO 3)\,dx<\infty.
  \end{equation*}
  Then there exists $u\in
  W^{2,2}_{\delta}(S,\R^3)$ such that
  \begin{align*}
    u^h-\fint_\Omega u^h\,dx&\to u,&\qquad&\text{strongly in }L^2(\Omega,\R^3),\\
    \nabla_h u^h&\to(\,\nabla'u,\,n\,)&\qquad&\text{strongly in
    }L^2(\Omega,\R^{3\times 3}),
  \end{align*}
  as $h\to 0$ after passing to subsequences and extending $u$ and $n$ trivially to $\Omega$.
\end{theorem}

Theorem \ref{T:main} and Theorem \ref{T:FJM-compactness} imply by standard arguments from the theory of
$\Gamma$-convergence that minimizers of functionals of the form
$$
h^{-2}\mathcal E^{h,\eh}(\cdot)+\text{ ``(rescaled) dead loads''},
$$
subject to certain boundary conditions, converge to minimizers of
$$
\mathcal E_\gamma(\cdot)+\text{ ``dead loads''},
$$
subject to certain
boundary conditions. For  details see \cite{FJM2}.

In the special case when $W(y,F)=W(F)$ is homogeneous Theorem~\ref{T:main}
reduces to the result in \cite{FJM1}. The proof
 of our main result emulates their argument as far as possible.
\\
We now explain our approach. The bending regime is a borderline case in the hierarchy of plate
models. On one hand it allows for large deformations, on the other hand it corresponds to small
strains: By Theorem~\ref{T:FJM-compactness} a sequence $\{u^h\}_{h>0}$ with \textit{finite bending
  energy}
in general converges to a non-trivial deformation.
However, the associated non-linear strain $\sqrt{(\nabla_hu^h)^t(\nabla_h
  u^h)}-\id$ converges to zero. Indeed, let
\begin{equation}\label{eq:strain}
  E^h:=\frac{\sqrt{(\nabla_hu^h)^t(\nabla_h u^h)}-\id}{h}
\end{equation}
denote the \textit{scaled non-linear strain} associated with $u^h$. Then due
to the elementary inequality $\left| \sqrt{F^TF}-\id \right| \leq
\dist(F,\SO 3)$ we find that $\{E^h\}_{h>0}$ is bounded in $L^2$
when $\{u^h\}_{h>0}$ has finite bending energy.

The smallness of the nonlinear strain is crucial for our extension to simultaneous homogenization and dimension reduction: By \eqref{ass:expansion} the elastic energy is related
to the
nonlinear strain in a quadratic way -- indeed,  we formally have
\begin{equation}\label{Neu:eq:3}
  \frac{1}{h^2}\mathcal E^{h,\e}(u^h)\approx \int_\Omega Q(\frac{x'}{\e},E^h(x))\,dx.
\end{equation}
Heuristically, the right-hand side is obtained by linearizing the
stress-strain relation, while preserving the geometric non-linearity.

Due to the convexity of
the right-hand side in \eqref{Neu:eq:3} only oscillations of
$\{E^h\}_{h>0}$ that emerge precisely on scale $\e$ are relevant for
homogenization. A tool to describe such oscillations is
two-scale convergence. In Section~\ref{S:two-scale} we characterize (partially)
the possible two-scale limits of $\{E^h\}_{h>0}$. This is
the main ingredient for the lower bound in Theorem~\ref{T:main}.
\\
Assume $u^h$ converges to some bending deformation with second fundamental form $\secf$. Then
any two-scale accumulation point of $\{E^h\}_{h>0}$ can be written in the form
\begin{equation}\label{Neu:eq:5}
  x_3
\left(
  \begin{array}{cc}
    \secf(x') & \begin{array}{c}0\\0\end{array}\\
    0\;\;\;\;0 &0
  \end{array}\right) + \widetilde E(x,y)
\end{equation}
where $\widetilde E:\Omega\times Y\to\R^{3\times 3}_{\sym}$ is a relaxation
field that captures oscillations and is a priori ``unknown''. In
Proposition~\ref{P:strain} we prove that $\widetilde E$ has to be of
specific form.
The $\Gamma$-limit of $h^{-2}\mathcal E^h$ is then
obtained by relaxation:
\begin{equation*}
  \inf_{\widetilde E}\int_\Omega\int_Y Q\left(y,  x_3
{\scriptsize\left(
  \begin{array}{cc}
    \secf(x') & \begin{array}{c}0\\0\end{array}\\
    0\;\;\;\;0 &0
  \end{array}\right)} + \widetilde E(x,y)\,\right)\,dy\,dx,
\end{equation*}
where the infimum is taken over all $\tilde{E}$ of the specific form given in Proposition \ref{P:strain}.

We conclude this section by discussing the dependency of our limiting model on the parameter $\gamma$, which describes the relative
scaling between $h$ and $\e$.
The relaxed quadratic form $Q_{2,\gamma}$ continuously depends on
$\gamma$. In fact, with  \cite[Lemma~5.2]{NeuVel} at hand one can
easily identify the limits
\begin{equation*}
  \lim\limits_{\gamma\to
    0}Q_{2,\gamma}(A)\qquad\text{and}\qquad\lim\limits_{\gamma\to\infty}Q_{2,\gamma}(A)\qquad (A\in\R^{2\times2}),
\end{equation*}
which  yield proper quadratic forms on $\R^{2\times 2}$ that
vanish on skew-symmetric matrices and are positive definite on
symmetric matrices. In particular, the limit for $\gamma\to\infty$
coincides with $Q_{2,\infty}$. The limit for $\gamma\to 0$ can be
identified as well: We introduce the dimension reduced quadratic
form $Q_2(y,A)$ for all $A\in\R^{2\times 2}$ via
\begin{equation*}
  Q_2(y,A)=\min_{d\in\R^3}Q\left(y,\sum_{\alpha,\beta=1}^2A_{\alpha\beta}(e_\alpha\otimes
    e_\beta)+d\otimes e_3\right).
\end{equation*}
Then $Q_{2,\gamma}(A)$ converges for $\gamma\to 0$ to
\begin{eqnarray*}
  Q_{2,0}(A)&:=&
  \inf_{B,\zeta,\varphi}\iint_{I\times Y}Q_2\left(
    y,\;A+x_3B+\sym(\nabla_y\zeta+x_3\nabla_y^2\varphi)\right)\,dy\,dx_3
\end{eqnarray*}
where the infimum is taken over all $B\in\R^{2\times 2}_{\sym}$,
$\zeta\in H^1(\mathcal Y,\R^2)$ and $\varphi\in H^2(\mathcal Y)$.

A similar behavior has been observed in \cite{Neukamm-12, NeuVel}
where also the case $\gamma=0$ is considered (for rods and von-K\'arm\'an
plates, respectively). In the von-K\'arm\'an case (see \cite{NeuVel}) it turns out
that in the regime $h\ll\eh$ the limit $\gamma\to 0$ of the quadratic energy
density indeed recovers the energy density obtained via
$\Gamma$-convergence. It
is not clear whether or not this picture extends to the bending regime.

\section{Two-scale limits of the nonlinear strain}\label{S:two-scale}

Two-scale convergence was introduced in \cite{Nguetseng-89,
  Allaire-92} and has been extensively applied to various problems in
homogenization. In this article we work with the following variant of
two-scale convergence which is adapted to dimension reduction.
\begin{definition}[two-scale convergence]
  \label{D:two-scale-thin}
  We say a bounded sequence $\{f^h\}_{h>0}$ in $L^2(\Omega)$ two-scale
  converges to $f\in L^2(\Omega\times Y)$ and we write $f^h\wtto f$, if
  \begin{equation*}
    \lim\limits_{h\to 0}\int_\Omega
    f^h(x)\psi(x,\frac{x'}{\eh})\,dx=\iint_{\Omega\times Y} f(x,y)\psi(x,y)\,dy\,dx
  \end{equation*}
  for all $\psi\in C^\infty_0(\Omega,C(\mathcal Y))$. When
  $||f^h||_{L^2(\Omega)}\to||f||_{L^2(\Omega\times Y)}$ in addition,
  we say that $f^h$ strongly two-scale converges to $f$ and write
  $f^h\stto f$. For vector-valued functions, two-scale convergence is defined componentwise.
\end{definition}

Since we identify functions on $S$ with their trivial
extension to $\Omega$, the definition above contains the standard
notion of two-scale convergence on $S\times Y$ as a special case. Indeed, when $\{f^h\}_{h>0}$ is a sequence in
$L^2(S)$, then $f^h\wtto f$ is equivalent to
\begin{equation*}
  \lim\limits_{h\to 0}\int_S
  f^h(x')\psi(x',\frac{x'}{\eh})\,dx'=\iint_{S\times Y} f(x',y)\psi(x',y)\,dy\,dx'
\end{equation*}
for all  $\psi\in C^\infty_0(S,C(\mathcal Y))$.

The main ingredient in the proof of the lower bound part of
Theorem~\ref{T:main} is the following characterization of the possible
two-scale limits of nonlinear strains.

\begin{proposition}
\label{P:strain}
Let $\{u^h\}_{h>0}$ be a sequence of deformations with finite bending
energy, let $u\in\WW(S,\R^3)$ with second fundamental form $\secf$, and assume that
\begin{align*}
  u^h-\fint_\Omega u^h\,dx&\to u&\qquad&\text{strongly in
  }L^2(\Omega,\R^3),\\
  E^h:=\frac{\sqrt{(\nabla_h u^h)^t\nabla_h u^h}-\id}{h}&\wtto: E&\qquad&\text{weakly two-scale}
\end{align*}
for some $E\in L^2(\Omega\times Y; \R^{3\times 3})$.
\begin{enumerate}[(a)]
\item If $\gamma\in(0,\infty)$ then
  there exist
  $
  B \in L^2(S, \R^{2\times 2}_{\sym}),
  $
  and $\phi\in L^2(S,\mathring H^1(I \times \mathcal Y,\R^3))$ such that
  \begin{equation}\label{P:strain:1a}
    E(x, y) =
    \left(
      \begin{array}{cc}
        x_3\secf(x')+B(x') & \begin{array}{c}0\\0\end{array}\\
        0\;\;\;\;0 &0
      \end{array}\right)
    +
    \sym\left (\nabla_y \phi(x, y)\,,\,\tfrac{1}{\gamma}
      \partial_3 \phi(x, y) \right).
  \end{equation}
\item If $\gamma=\infty$ then there exist
  $B \in L^2(S, \R^{2\times 2}_{\sym})$, $\phi\in L^2(\Omega,\mathring
  H^1(\mathcal Y,\R^3))$, and $d\in L^2(\Omega,\R^3)$ with
  \begin{equation}\label{P:strain:1b}
    E(x, y) =
    \left(
      \begin{array}{cc}
        x_3\secf(x')+B(x') & \begin{array}{c}0\\0\end{array}\\
        0\;\;\;\;0 &0
      \end{array}\right)
    +
    \sym\left (\nabla_y \phi(x, y)\,,\,d(x) \right).
  \end{equation}
\end{enumerate}
\end{proposition}
\begin{remarks}
  \begin{enumerate}[(i)]
  \item In \cite{FJM1} a coarser characterization of possible weak limits of $\{E^h\}_{h>0}$ was obtained: In the
    situation of the previous proposition let $(E^h)'$ denote the $2\times 2$
    matrix obtained from $E^h$ by deleting the third row and column. Then it was shown in \cite{FJM1} that
    \begin{equation*}
      (E^h)'\wto x_3\secf(x')+B'(x')\qquad\text{weakly in
      }L^2(\Omega,\R^{2\times 2}_{\sym})).
    \end{equation*}
    Proposition~\ref{P:strain} refines this by capturing, in addition, oscillations on the scale $\e$.
  \item Proposition~\ref{P:strain} still only yields an incomplete characterization of the possible
    structure of the two-scale limiting strain $E$: it is not true that every $E$ in the form of
    \eqref{P:strain:1a} (resp. \eqref{P:strain:1b}) can be recovered as a two-scale limit of a
    sequence of nonlinear strains. For instance, when $u$ is affine, i.~e. $\secf=0$, then
    not every two-scale limiting strain of
    the form \eqref{P:strain:1a} with $B$ arbitrary and $\phi=0$ can
    emerge. 

    In our construction of recovery sequences a special role is played
    by the matrix $B$, which is ``recovered'' by corrections
    of the isometry of order $h$. More precisely, these corrections are obtained by
    solving the  equation
    \begin{equation}\label{eq:14}
      B=\sym\nabla'g+\alpha\secf\qquad\text{for }g:S\to\R^d\text{ and }\alpha:S\to\R.
    \end{equation}
    As shown in \cite[Lemma 3.3]{Schmidt-07}, equation \eqref{eq:14} can be solved locally on regions where  $\secf\neq 0$, provided that $u$ is smooth.
On the level of these ``order $h$ corrections'' the deformed
    plate behaves like a shell and the condition $\secf\neq 0$ corresponds
    to the property that the shell is developable without affine region.

    An important observation is that, in spite of not giving an exhaustive characterization of limiting strains,
    the result of Proposition~\ref{P:strain} is just sharp
    enough to obtain the optimal lower bound for $h^{-2}\mathcal E^{h,\eh}$. This is because
    on regions where $\secf=0$, corrections associated to $B$  can be ignored, since they do not
    reduce the energy (as $Q(y,F)$ is minimal for $F=0$).

    In contrast, for rods and von-K\'arm\'an
    plates, exhaustive characterizations were obtained in \cite[Theorem~3.5]{Neukamm-12} and \cite[Proposition~3.3]{NeuVel}.

\item A key technical ingredient in the proof of Proposition \ref{P:strain} is Lemma \ref{L:two-scale-pc} below.
It allows us to work with piecewise constant $SO(3)$-valued approximations of the deformation gradient,
as opposed to smooth $SO(3)$-valued approximations. The latter were used in the proof of the 1d case given in
\cite{Neukamm-12}. In the 2d case studied here, the use of such a smooth $SO(3)$-valued approximation
would require {\em small} limiting energy, cf. \cite[Remark 5]{FJM2}. Thanks to Lemma \ref{L:two-scale-pc},
our result is not restricted to small limiting energy. Incidentally,
the use of this lemma also simplifies the proof of the convergence
statement in the 1d case.
\end{enumerate}
\end{remarks}

The starting point of the proof of the previous Proposition is  \cite[Theorem 6]{FJM2},
which we combine with the last remark in \cite[Section 3]{FJM1} in order to allow for $\gamma_0 < 1$.

\begin{lemma}
\label{t6f2}
Let $\gamma_0\in (0, 1]$ and let $h,\delta>0$ with $\gamma_0\leq \frac{h}{\delta}\leq\frac{1}{\gamma_0}$. There exists a constant $C$, depending only
on $S$ and $\gamma_0$, such that the following is true:
if $u\in H^1(\Omega, \R^3)$ then there exists a map $R : S\to SO(3)$
which is piecewise constant on each cube $x + \delta Y$ with $x\in \delta\Z^2$ and there exists
$\t{R}\in H^1(S, \R^{3\times 3})$ such that
$$
\|\nabla_hu - R\|^2_{L^2(\Omega)} + \|R - \t{R}\|^2_{L^2(S)} + h^2\|\nabla\t{R}\|^2_{L^2(S)}
\leq C \|\dist(\nabla_h u, \SO 3) \|^2_{L^2(\Omega)}.
$$
\end{lemma}

Let us recall some well-known properties of two-scale convergence. We
refer to \cite{Allaire-92, Visintin-06, Mielke-Timofte-07} for proofs
in the standard two-scale setting and to \cite{Neukamm-10} for the
easy adaption to the notion of two-scale convergence considered here.
\begin{lemma}\label{kompts}
  \begin{enumerate}[(i)]
  \item Any sequence that is bounded in $L^2(\Omega)$ admits a two-scale
    convergent subsequence.
  \item Let $\t f\in L^2(\Omega\times Y)$ and let $f^h\in L^2(\Omega)$ be such that $f^h\wtto \t{f}$.
   Then $f^h\wto\int_Y\t f(\cdot,y)\,dy$ weakly in $L^2(\Omega)$.
  \item Let $f^0$ and $f^h\in L^2(\Omega)$ be such that $f^h\wto
    f^0$ weakly in
    $L^2(\Omega)$. Then (after passing to subsequences) we have $f^h\wtto f^0(x)+\t f$ for some $\t f\in L^2(\Omega\times
    Y)$ with $\int_Y\t f(\cdot,y)\,dy=0$ almost everywhere in $S$.
  \item Let $f^0$ and $f^h\in H^1(\Omega)$ be such that $f^h\to f^0$ strongly in
    $L^2(\Omega)$. Then $f^h\stto f^0$, where we extend $f^0$ trivially to $\Omega\times Y$.
  \item Let $f^0$ and $f^h\in H^1(S)$ be such that $f^h\wto f^0$ weakly in $H^1(S)$.
    Then (after passing to subsequences)
    \begin{equation*}
      \nabla' f^h\wtto \nabla 'f^0+\nabla_y\phi
    \end{equation*}
    for some $\phi\in L^2(S, H^1(\mathcal Y))$.
  \end{enumerate}
\end{lemma}

The following lemma is an immediate consequence of Lemma \ref{kompts}, cf.
\cite[Theorem~6.3.3]{Neukamm-10} for a proof.

\begin{lemma}\label{L:two-scale}
  Let $u^0$ and $u^h\in H^1(\Omega,\R^3)$ be such that $u^h\wto u^0$ weakly in
  $H^1(\Omega,\R^3)$.
  \begin{enumerate}[(a)]
  \item If $\gamma\in(0,\infty)$ then there exists
$\phi\in L^2(S,\mathring H^1(I\times\mathcal Y,\R^3))$ such that
(after passing to subsequences)
    \begin{equation*}
      \nabla_h u^h\wtto (\nabla'u^0\,,\,0)+(\nabla_y\phi\,,\,\tfrac{1}{\gamma}\partial_3\phi).
    \end{equation*}
  \item If $\gamma=\infty$ then there exist
 $\phi\in L^2(S,\mathring H^1(I\times\mathcal Y,\R^3))$ and
    $d\in L^2(\Omega,\R^3)$ such that (after passing to subsequences)
    \begin{equation*}
      \nabla_h u^h\wtto (\nabla'u^0\,,\,0)+(\nabla_y\phi\,,\,d).
    \end{equation*}
      \end{enumerate}
\end{lemma}

At several places in our proof we will need to make sense of a two-scale limit
for sequences which might be unbounded in $L^2$, but which nevertheless have controlled oscillations
on the scale $\e$. In order to capture these oscillations, we `renormalize' the sequence by
throwing away the (divergent) part which does not oscillate on the scale $\e$. (For bounded sequences,
this latter part gives rise to the weak limit, but the point here is that our sequences may be unbounded.)
Equivalently, we weaken the notion of two-scale convergence by restricting the admissible test functions
to functions with vanishing cell average.
\\

More precisely, for a sequence $\{f^h\}_{h>0}\subset L^2(\Omega)$ and $\t f\in
  L^2(\Omega\times Y)$ with $\int_Y\t f(\cdot,y)\,dy=0$ almost everywhere in $\Omega$, we write
  $$
  f^h \otto \t f
  $$
  if
\begin{multline}\label{eq:ren}
  \lim\limits_{h\to 0}\int_\Omega
  f^h(x)\varphi(x)g(\tfrac{x'}{\eh})\,dx=\iint_{\Omega\times Y} \t f(x,y)\varphi(x)g(y)\,dy\,dx\\
  \text{for all $\varphi\in C^\infty_0(\Omega)$ and $g\in C^{\infty}(\mathcal
    Y)$ with $\int_Yg\,dy=0$.}
\end{multline}

The proof of the following lemma is straightforward.

\begin{lemma}
  \label{L:osc}
  Let $f^0$ and $f^h\in L^2(\Omega)$ be such that $f^h\wto f^0$ weakly in
  $L^2(\Omega)$ and $f^h\otto \t f$. Then $f^h\wtto f^0+\t f$ weakly two-scale.
\end{lemma}

For the proof of Proposition~\ref{P:strain} we have to identify the
oscillatory part of two-scale limits for renormalized functions of the
form $\frac{1}{\eh}f^h$ where $f^h$ is either a sequence bounded in
$H^1(S)$ or piecewise affine with respect to  the lattice $\eh\Z^2$. The
following two lemmas treat these situations.

\begin{lemma}\label{L:two-scale-h1}
  Let $f^0$ and $f^h\in H^1(S)$ be such that $f^h\wto f^0$ weakly in
  $H^1(S)$ and assume that
  \begin{equation*}
    \nabla'f^h\wtto \nabla' f^0 + \nabla_y\phi
  \end{equation*}
  for some $\phi\in L^2(S,H^1(\mathcal Y))$ with
  $\int_Y\phi(\cdot,y)\,dy=0$ almost everywhere in $S$. Then
  \begin{equation*}
    \frac{f^h}{\eh}\otto \phi.
  \end{equation*}
\end{lemma}
\begin{proof}
  Since $f^h$ is independent of $x_3$, we must show that
  \begin{equation}\label{eq:2}
    \frac{1}{\eh}\int_Sf^h(x')g(\frac{x'}{\eh})\psi(x')\,dx'\to\iint_{S\times
      Y}\phi(x',y)g(y)\psi(x')\,dx'
  \end{equation}
  for all $g\in C^{\infty}(\mathcal Y)$ with $\int_Yg\,dy=0$ and $\psi\in
  C^\infty_0(S)$. For simplicity we write $\e$ instead of $\eh$. Let $G$ denote the unique solution in $C^2(\mathcal
  Y)$ to
  \begin{equation*}
    -\triangle_y G=g,\qquad \int_Y G\,dy=0.
  \end{equation*}
  Set $G^h(x'):=\e
  G(\frac{x'}{\e})$ so that
  \begin{equation*}
    \triangle G^h(x')=\frac{1}{\e} g(\frac{x'}{\e}).
  \end{equation*}
  Hence, the right-hand side of \eqref{eq:2} equals
  \begin{eqnarray}
    \notag
    \lefteqn{\int_S f^h\triangle G^h\psi\,dx'}&&\\
    \label{eq:3}
    &=&\int_S f^h\Big(\,\triangle (G^h\psi)-2\nabla
    G^h\cdot\nabla\psi-G^h\triangle \psi\,\Big)\,dx'\\
    \notag
    &=&-\int_S \nabla f^h\cdot\nabla (G^h\psi)\,dx'
    -2\int_S f^h(\nabla G^h\cdot\nabla\psi)\,dx'
    -\int_S f^h\,G^h\triangle \psi\,dx'.
  \end{eqnarray}
  By the chain rule and the definition of $G^h$ we have
  \begin{eqnarray*}
    \nabla (G^h\psi)(x')&=&\nabla
    G^h(x')\psi(x')+G^h(x')\psi(x')\\
    &=&\nabla_yG(\tfrac{x'}{\e})\psi(x')+\e G(\tfrac{x'}{\e})\psi(x').
  \end{eqnarray*}
  Since the right-hand side strongly two-scale converges to
  $\nabla_yG(y)\psi(x)$, and because $\nabla f^h\wtto \nabla
  f^0(x)+\nabla_y\phi(x,y)$ by assumption, we deduce that
  \begin{eqnarray*}
    \lefteqn{-\int_S \nabla f^h\cdot\nabla (G^h\psi)\,dx'}&&\\
    &\to& -\iint_{S\times
      Y}\Big(\,\nabla
    f^0(x')+\nabla_y\phi(x',y)\,\Big)\cdot\Big(\,\nabla_yG(y)\psi(x')\,\Big)\,dy\,dx'\\
    &&=\iint_{S\times Y}\phi(x',y)\triangle_yG(y)\psi(x')\,dy\,dx'\\
    &&=\iint_{S\times Y}\phi(x',y)g(y)\psi(x')\,dy\,dx'.
  \end{eqnarray*}
  Hence it suffices to show that the second and third
  integral on the right-hand side of \eqref{eq:3} vanish for $h\to 0$.
  We treat the second integral. Since $\nabla G^h(x')=\nabla_y
  G(\tfrac{x'}{\e})$ strongly two-scale converges to
  $\nabla_y G(y)$, and because $f^h\nabla\psi\to f^0\nabla\psi$ strongly in
  $L^2(S)$, we deduce that
  \begin{eqnarray*}
    -2\int_S f^h(\nabla G^h\cdot\nabla\psi)\,dx'\to -\iint_{S\times
      Y}f^0(x')\nabla\psi(x')\cdot\nabla_y G(y)\,dy\,dx'=0.
  \end{eqnarray*}
  The third integral on the right-hand side of \eqref{eq:3} vanishes
  simply because $f^h\triangle\psi$ is bounded in $L^2(S)$ and $G^h\to
  0$ in $L^2(S)$.
\end{proof}

\begin{lemma}
  \label{L:two-scale-pc}
  Let $f^0$ and $f^h\in L^\infty(S)$ be such that $f^h\stackrel{*}{\wto} f^0$
  weakly-* in $L^{\infty}(S)$. Assume that $f^h\in L^{\infty}(S)$ is constant on each cube $x +\eh Y,\ \
  x\in\eh\Z^2$. Then we have
  \begin{equation}\label{eq:4}
    \frac{1}{\eh} \int_S f^h(x')\psi(x') g\left( \frac{x'}{\eh}
    \right)\ dx' \to \int_S f^0(x') \nabla' \psi(x')\ dx'\cdot \int_Y g(y) y\ dy
  \end{equation}
  for all $g\in C(\mathcal Y)$ with $\int_Y g=0$ and $\psi\in C^\infty_0(S)$.
  In particular, if $f^0\in W^{1,2}(S)$ we have
  $$
\frac{1}{\eh}f^h\otto - (y\cdot\nabla') f^0.
$$
Here we write
$$
(y\cdot\nabla')f^0(x') = \sum_{\alpha=1,2}y_\alpha\partial_\alpha f^0(x').
$$
\end{lemma}
\begin{proof}
  We first argue that \eqref{eq:4} combined with $f^0\in
  W^{1,\infty}(S)$ implies the convergence of $\frac{1}{\eh}f^h$.
  Indeed, since $f^h$ is independent of $x_3$ it suffices to consider
  test functions $g$ and $\psi$ as in identity \eqref{eq:4}. Now the statement simply follows from the observation that the right-hand side of \eqref{eq:4} becomes
  \begin{eqnarray*}
    -\iint_{S\times Y}(y\cdot\nabla')f^0(x')\psi(x')g(y)\ dy\ dx'
  \end{eqnarray*}
  by an integration by parts. We prove \eqref{eq:4}. For simplicity we
  write $\e$ instead of $\eh$. We denote by $\t{\psi}^h$ an approximation
  of $\psi$ that is constant on each of the cubes $\xi+\e Y$,
  $\xi\in\e\Z^2$, say $\t\psi^h(x):=\psi(\xi_x)$ where $\xi_x\in\e\Z^2$
  denotes the cube $\xi_x+\e Y$ in which $x$ lies.
  Then we have
  \begin{align}\label{renorm-1}
    \int_S \frac{f^h(x)}{\e}\psi(x) g\left( \frac{x}{\e} \right)\ dx
    &= \int_S f^h(x) \frac{\psi(x) - \t{\psi}^h(x)}{\e} g\left( \frac{x}{\e} \right)\ dx
  \end{align}
  because
  $$
  \int_S f^h(x)\t{\psi}(x) g\left( \frac{x}{\e} \right)\ dx = 0,
  $$
  since $g$ has zero average over $Y$, and $f^h$ and $\t{\psi}^h$ are both piecewise constant.
  Let us compute the right-hand side of \eqref{renorm-1}. As $\xi_x\in
  \e\Z^2$ and $g\in C(\mathcal Y)$, we have
  $$
  g\left( \frac{x - \xi_x}{\e} \right) =  g\left( \frac{x}{\e} \right),
  $$
  and see (after extending $f^h$ to $\R^2$ by zero)
  \begin{eqnarray*}
    \lefteqn{\int_S f^h(x) \frac{\psi(x) - \t{\psi}^h(x)}{\e} g\left(
        \frac{x}{\e} \right)\ dx}&&\\
    &=& \sum_{\xi\in\e\Z^d} f^h(\xi) \int_{\xi + \e Y} \frac{\psi(x) - \t{\psi}^h(\xi)}{\e} g\left( \frac{x}{\e} \right)\ dx
    \\
    &=& \sum_{\xi\in\e\Z^d} f^h(\xi) \int_{\xi + \e Y} \left(\int_0^1 \left( \nabla'\psi(\xi + t(x - \xi))\right)\cdot \frac{x - \xi}{\e}\right) g\left( \frac{x-\xi}{\e} \right)\ dx
    \\
    &=& \sum_{\xi\in\e\Z^d} f^h(\xi) \int_{\e Y} \left(\int_0^1 \left(
        \nabla'\psi(\xi + t x)\right)\cdot \frac{x}{\e}\right) g\left( \frac{x}{\e} \right)\ dx
    \\
    &=& \sum_{\xi\in\e\Z^d} f^h(\xi)\e^2 \int_Y \left( \int_0^1 \nabla'\psi(\xi + t\e y)\cdot y \right) g(y)\ dy
    \\
    &=& \sum_{\xi\in\e\Z^d} f^h(\xi)\e^2 \int_Y \left( \int_0^1 \left( \nabla'\psi(\xi + t\e y) - \nabla'\psi(\xi) \right)\cdot y \right) g(y)\ dy
    \\
    &&+ \sum_{\xi\in\e\Z^d} f^h(\xi)\e^2 \int_Y \nabla'\psi(\xi)\cdot y g(y)\ dy.
  \end{eqnarray*}
  The first term on the right-hand side converges to zero as $h\to 0$ because
  $$
  \left| \nabla'\psi(\xi + t\e y) - \nabla'\psi(\xi) \right| \leq C\e
  $$
  for all $t\in [0, 1]$, simply because $\nabla'\psi$ is Lipschitz.
  \\
  Hence it remains to compute:
  \begin{align*}
    \sum_{\xi\in\e\Z^d} &\e^2 f^h(\xi) \nabla'\psi(\xi) \cdot \int_Y y g(y)\ dy
    \\
    &=
    \sum_{\xi\in\e\Z^d}\e^2 \left( f^h(\xi) \nabla'\psi(\xi) - \fint_{\xi +\e Y} f^h(z) \nabla'\psi(z)\ dz\right) \cdot \int_Y y g(y)\ dy
    \\
    &+ \sum_{\xi\in\e\Z^d}\e^2 \fint_{\xi +\e Y} f^h(z) \nabla'\psi(z)\ dz \cdot \int_Y y g(y)\ dy
    \\
    &=
    \sum_{\xi\in\e\Z^d}\e^2 \left( f^h(\xi) \nabla'\psi(\xi) - \fint_{\xi +\e Y} f^h(z) \nabla'\psi(z)\ dz\right) \cdot \int_Y y g(y)\ dy
    \\
    &+ \int_{\R^2} f^h(x) \nabla'\psi(x)\ dx \cdot \int_Y y g(y)\ dy.
  \end{align*}
  Since $\spt \psi\subset S$, the last term equals
$$
\int_S f^h(x) \nabla'\psi(x)\ dx \cdot \int_Y y g(y)\ dy.
$$
The claim follows because $f^h\stackrel{\star}{\wto} f^0$ and because
\begin{align*}
    \sum_{\xi\in\e\Z^d}\e^2 \left( f^h(\xi) \nabla'\psi(\xi) -
      \fint_{\xi +\e Y} f^h(z) \nabla'\psi(z)\ dz\right) \cdot \int_Y
    y g(y)\ dy\to 0
\end{align*}
as $h\to 0$. To see this, we compute recalling that $f^h(x) = f^h(\xi)$ for all $x\in \xi +\e Y$:
\begin{align*}
f^h(\xi)\nabla'\psi(\xi) - \fint_{\xi +\e Y} f^h(z) \nabla'\psi(z)\,dz &=
f^h(\xi) \left( \nabla'\psi(\xi) - \fint_{\xi +\e Y} \nabla'\psi(z)\,dz \right) \leq Ch,
\end{align*}
again because $\nabla'\psi$ is Lipschitz.

\end{proof}

\begin{proof}[Proof of Proposition~\ref{P:strain}, case $\gamma\in(0,\infty)$]
{\bf Step 1.\  } Without loss of generality we assume that all $u^h$ have average zero.
Theorem~\ref{T:FJM-compactness} then implies that
\begin{equation}\label{eq:5}
  \nabla_hu^h\to R:=(\nabla'u,\ n)\qquad\text{strongly in
  }L^2(\Omega,\R^{3\times 3})
\end{equation}
where $n$ denotes the normal to $u$.
Let $R^h$, $\t{R}^h$ be the maps obtained by applying Lemma~\ref{t6f2}
to $u^h$ with $\delta(h)=\eh$. Due to the uniform bound on $\nabla'\t R^h$ given by
Lemma~\ref{t6f2},  $R^h$ and $\t R^h$ are precompact in
$L^2(S,\R^{3\times 3})$.
Hence, \eqref{eq:5} combined with $||R^h-\nabla_h
u^h||_{L^2}\to 0$ (which also follows from Lemma~\ref{t6f2}) shows that $R^h$ and $\t R^h$ strongly converge in
$L^2(S,\R^{3\times 3})$ to $R$.
% \\
Following \cite{FJM1}, we introduce the approximate
 strain
\begin{equation}
  G^h(x)=\frac{(R^h)^t \nabla_h u^h(x)-\id}{h}.
\end{equation}
We set $\overline u^h(x') = \int_I u^h(x', x_3)\ dx_3$ and define $z^h\in H^1(\Omega, \R^3)$ via
\begin{equation}
  u^h(x', x_3) = \overline{u}^h(x') + hx_3 \t{R}^h(x') e_3 + h z^h(x', x_3).
\end{equation}
Then clearly $\int_I z^h(x', x_3) dx_3 = 0$ and we compute
\begin{equation}
  \label{combi-0}
  \frac{\nabla_hu^h - R_h}{h}=
  \left(\,\frac{ \nabla'\overline{u}^h - (R^h)'}{h} + x_3\nabla'\t{R}^he_3,\ \frac{1}{h} (\t{R}^he_3 - R^h e_3)\right) + \nabla_h z^h.
\end{equation}
For a given matrix $M\in \R^{3\times 3}$, we denote by $M'$ the
$3\times 2$-matrix obtained by deleting the third column. We use the notation $(y\cdot\nabla')R(x'):=y_1\partial_1R(x')+y_2\partial_2R(x')$

{\bf Step 2.\ } Let us for the moment take for granted that
there exist $B'\in L^2(S,\R^{3\times 2})$, $\t z\in
L^2(S,H^1(I\times\mathcal Y,\R^3))$, $\t v,\,\t w\in
  L^2(S,H^1(\mathcal Y,\R^3))$ and $w^0\in L^2(S,\R^3)$,
such that, after passing to a subsequence,
\begin{eqnarray}
  \label{P:strain:eq1}
  \nabla_hz^h&\wtto&  (\nabla_y \t{z}\,,\,\frac{1}{\gamma}\partial_3 \t{z}),\\
  \label{P:strain:eq2}
  \frac{\nabla'\overline{u}^h-(R^h)'}{h}&\wtto& B'(x')+\frac{1}{\gamma}(y\cdot\nabla')R'(x') +\nabla_y\t v(x',y),\\
  \label{P:strain:eq3}
  x_3\nabla'\t{R}^he_3&\wtto& x_3\nabla' R(x')e_3 + x_3\nabla_y\t w(x',y),\\
  \label{P:strain:eq4}
  \frac{1}{h} (\t{R}^he_3 - R^h e_3)&\wtto& \frac{1}{\gamma}(y\cdot\nabla')R(x')e_3
 + w^0(x') +\frac{1}{\gamma}\t w(x',y).
\end{eqnarray}
We now proceed to prove that the proposition follows from these convergences.
\\
First notice that it suffices to identify the
symmetric part of the two-scale limit $G$ of the sequence $G^h$. Indeed, since
$\sqrt{(\id+h F)^t(\id +h F)}=\id + h\sym F$ up to terms of higher
order, the convergence $G^h\wtto G$
implies $E=\sym G$ (see e.g. \cite[Lemma~4.4]{Neukamm-12} for a
proof).
\\
We now identify $\sym G$. By combining \eqref{P:strain:eq1} -- \eqref{P:strain:eq4} with
identity \eqref{combi-0}, we
find that $R^hG^h$ weakly two-scale converges to
\begin{equation}\label{P:strain:eq5}
  (B',\,0)
  +
  \Bigg(\,x_3\nabla'R(x')e_3 \,,0\,\Bigg)
  +
  \Bigg(\,\nabla_y\t\phi,\,\frac{1}{\gamma}\partial_3\t\phi\,\Bigg)
  +
  (y\cdot\nabla')R(x')
\end{equation}
where
\begin{equation*}
  \t\phi(x,y):=\t z(x,y) + \t v(x',y) +x_3\t w(x',y) + x_3w^0(x').
\end{equation*}
Due to the strong $L^2$-convergence  $R^h\to R$, we deduce that
$G^h$ weakly two-scale converges to \eqref{P:strain:eq5} multiplied  with $R^t$
from the left. The first and second term yield
\begin{equation*}
  \left(
    \begin{array}{cc}
      \t B(x')+ x_3\secf(x')&
      \begin{array}{c}
        0\\0
      \end{array}\\
      \begin{array}{cc}
        b_1(x')\ \ \ b_2(x')
      \end{array}&0
    \end{array}
\right),
\end{equation*}
where $\t B$ denotes the $2\times 2$-matrix obtained by deleting the
third column of $R^tB'$ and $(b_1,b_2)$ are defined as the entries of the third
row of $R^tB'$.
Upon left multiplication by $R^t$, the last term in \eqref{P:strain:eq5}
yields a skew-symmetric term. Thus we have shown:
\begin{equation*}
  \sym G(x,y) =
\left(
    \begin{array}{cc}
      \sym\t B + x_3 \secf&
      \begin{array}{c}
        0\\0
      \end{array}\\
      0\quad 0&0
    \end{array}
  \right)
  +\sym\Bigg(\,\nabla_y\phi,\,\frac{1}{\gamma}\partial_3\phi\,\Bigg)
\end{equation*}
where
\begin{equation*}
  \phi(x,y):=R^t(x')\t\phi(x,y)+\gamma x_3\left(
    \begin{array}{c}
      b_1(x')\\b_2(x')\\0
    \end{array}
\right).
\end{equation*}
% %

{\bf Step 3.\ } It remains to prove \eqref{P:strain:eq1} -- \eqref{P:strain:eq4}.
Since $\nabla_hz^h$ is uniformly bounded in $L^2$ and since
$\int_I z^h\ dx_3 = 0$ by construction, \eqref{P:strain:eq1} directly follows
from  Lemma \ref{L:two-scale}.
\\
Next we prove \eqref{P:strain:eq2}. By Lemma~\ref{t6f2} and
Lemma~\ref{kompts} (i) there exists $V\in L^2(S\times Y, \R^{3\times 2})$ such that (after taking subsequences)
\begin{equation}
  \label{2scq}
  \frac{\nabla'\overline{u}^h - (R^h)'}{h}\wtto V.
\end{equation}
Let us verify that
\begin{equation*}
  V(x',y)=B'(x')+ (y\cdot\nabla')R'(x') +\nabla_y\t v(x',y),
\end{equation*}
where $B':=\int_Y V(\cdot,y)\,dy$ and $\t v\in L^2(S,H^1(\mathcal Y))$.
This is equivalent to showing that
\begin{multline}\label{P:strain:eq7}
  \iint_{S\times Y}V(x',y):(\nabla_y^\perp G)(y)\psi(x')\,dy\,dx'\\
  =
  \iint_{S\times Y}(y\cdot\nabla')R'(x'):(\nabla_y^\perp G)(y)\psi(x')\,dy\,dx'
\end{multline}
for all $G\in C^1(\mathcal Y,\R^3)$, $\psi\in C^\infty_0(S)$. (Here
and below $\nabla^\perp_y:=(-\partial_{y_2},\partial_{y_1})$).
% %
Set $G^h(x'):=\eh G(\tfrac{x'}{\eh})$, so that $(\nabla')^\perp
G^h(x')=(\nabla_y^\perp G)(\frac{x'}{\eh})$.
\\
To prove \eqref{P:strain:eq7}, note that since $\int_S \nabla'\overline{u}^{h} : {\nabla'}^{\perp}(G^h\psi) = 0$, we have
\begin{eqnarray*}
  \lefteqn{\int_S \frac{\nabla'\overline{u}^{h}(x')}{h} : (\nabla^{\perp}_y
  G)\left(\tfrac{x'}{\eh}\right)\psi(x')\ dx'}&&\\\notag
  &=&
  \int_S \frac{\nabla'\overline{u}^{h}(x')}{h} : {\nabla'}^\perp
  G^h(x')\psi(x')\ dx' \\\notag
  &=&-\int_S \frac{\nabla'\overline{u}^{h}}{h} : G^h(x')\otimes
  {\nabla'}^{\perp}\psi(x')\ dx'\\\notag
  &=&
  -\frac{\eh}{h}\int_S \nabla'\overline{u}^{h} : G(\tfrac{x'}{\eh})\otimes
  {\nabla'}^{\perp}\psi(x')\ dx'.
\end{eqnarray*}
The right-hand side converges to $0$, since
$\frac{\eh}{h}\nabla'\overline{u}^h$ strongly converges in $L^2$ and $G(\frac{\cdot}{\eh})\wto 0$
weakly in $L^2$. In addition, Lemma~\ref{L:two-scale-pc} yields
\begin{equation}
\label{combi-8}
\frac{R^h}{h}=\frac{\eh}{h}\frac{1}{\eh}R^h \otto \frac{1}{\gamma}(y\cdot\nabla')R(x'),
\end{equation}
and thus
\begin{multline*}
  \int_S \frac{\nabla'\overline{u}^{h}(x') - (R^h)'(x')}{h} : (\nabla^{\perp}_y G)\left(\frac{x'}{\eh}\right)\psi(x')\ dx'\\
  \to-\iint_{S\times Y} \frac{1}{\gamma}(y\cdot\nabla')R'(x') : \nabla^{\perp}_y G(y)\,\psi(x')\ dy\,dx'.
\end{multline*}
On the other hand, the left-hand side converges to
$$
\iint_{S\times Y} V(x', y) : \nabla_y^{\perp}G(y)\ \psi(x')\ dy\ dx'.
$$
Hence, \eqref{P:strain:eq7} and thus \eqref{P:strain:eq2} follows.

We prove \eqref{P:strain:eq3} and \eqref{P:strain:eq4}. By Lemma~\ref{t6f2} the right-hand side
in \eqref{P:strain:eq4} is uniformly bounded in $L^2(S,\R^{3})$ and
thus we have (after passing to subsequences)
\begin{equation*}
  \frac{(\t R^h-R^h)e_3}{h}\wtto w(x',y)
\end{equation*}
for some $w\in L^2(S\times Y,\R^3)$. Set $w^0(x'):=\int_Y w(x',y)\,dy$.
Since $\t{R}^he_3\wto R e_3$ weakly in $H^1(S, \R^3)$,
we know from Lemma \ref{kompts} (v) that there exists $\t{w}\in L^2(S, H^1(I\times Y, \R^3))$
such that
\begin{equation}
\label{combi-2}
\nabla' \t{R}^h e_3 \wtto \nabla'Re_3 + \nabla_y\t{w}.
\end{equation}
This implies \eqref{P:strain:eq3}. The combination of \eqref{combi-2} with
Lemma \ref{L:two-scale-h1} yields $\frac{\t{R}^he_3}{h}\otto \gamma^{-1}\t{w}$.
Together with \eqref{combi-8} we get
$$
\frac{(\t{R}^h - R^h)e_3}{h}\otto \gamma^{-1}\t{w}(x',y) + \gamma^{-1}(y\cdot\nabla') R(x') e_3,
$$
and  \eqref{P:strain:eq4} follows from Lemma~\ref{L:osc}.

\end{proof}

\begin{proof}[Proof of Proposition~\ref{P:strain}, case
  $\gamma=\infty$]
  The argument is similar to the case $\gamma\in(0,\infty)$. Therefore
  we only indicate the required modifications. Step 1 (of the
  proof for $\gamma\in(0,\infty)$) holds verbatim modulo the following
  change: As a difference to $\gamma\in(0,\infty)$, in the case
  $\gamma=\infty$ we set $\delta(h):=\lceil\tfrac{h}{\eh}\rceil \eh$
  where $\lceil s\rceil$ denotes the smallest, positive integer larger
  or equal to $s$. By construction $\delta(h)$ is an integer multiple of
  $\eh$ and we have $\frac{h}{\delta(h)}\sim 1$. Hence,
  Lemma~\ref{t6f2} yields maps $R^h$ and $\t R^h$ with bounds uniform
  in $h$, and moreover $R^h$ is constant on each cube $x+\eh Y$,
  $x\in\eh\Z^2$.

  Similar to Step 2 (of the proof for $\gamma\in(0,\infty)$) the statement
  of the proposition can be reduced to show that (up to subsequences)
\begin{eqnarray}
  \label{P:strain:eq1b}
  \nabla_hz^h&\wtto&  (\nabla_y \t{z}\,,\,d'),\\
  \label{P:strain:eq2b}
  \frac{\nabla'\overline{u}^h-(R^h)'}{h}&\wtto& B'(x') +\nabla_y\t v(x',y),\\
  \label{P:strain:eq3b}
  x_3\nabla'\t{R}^he_3&\wtto& x_3\nabla' R(x')e_3 + x_3\nabla_y\t w(x',y),\\
  \label{P:strain:eq4b}
  \frac{1}{h} (\t{R}^he_3 - R^h e_3)&\wtto& w^0(x').
\end{eqnarray}
where $R$ is defined as in the case $\gamma\in(0,\infty)$, and $B'\in L^2(S,\R^{3\times 2})$, $\t z\in
L^2(\Omega, \mathring H^1(\mathcal Y,\R^3))$, $d'\in L^2(\Omega,\R^3)$, $\t v,\,\t w\in
  L^2(S,H^1(\mathcal Y,\R^3))$ and $w^0\in L^2(S,\R^3)$. Indeed, by
  the same arguments as for $\gamma\in(0,\infty)$,
  \eqref{P:strain:eq1b} --   \eqref{P:strain:eq4b} imply that
  \begin{equation*}
    \sym G(x,y)=
    \left(
      \begin{array}{cc}
        \sym\t B + x_3 \secf&
        \begin{array}{c}
          0\\0
        \end{array}\\
        0\quad 0&0
      \end{array}
    \right)
    +
    \sym\left (\nabla_y \phi(x, y)\,,\,d(x) \right),
  \end{equation*}
  where
  \begin{equation*}
    \phi(x,y):=R^t(\t z+\t v+x_3\t w),\qquad d=R^td'+R^tw^0 +\left(\begin{array}{c} b_1 \\ b_2 \\0 \end{array} \right),
  \end{equation*}
  and $\t B$, $R$ and $(b_1,b_2)$ are defined as in the case
  $\gamma\in(0,\infty)$.

  The proof of   \eqref{P:strain:eq1b} --   \eqref{P:strain:eq4b} is
  similar to Step~3 of the proof for the case $\gamma\in(0,\infty)$.
\end{proof}

\section{Proof of Theorem~\ref{T:main}}\label{S:proof}

As a preliminary step we need to establish some continuity properties of the quadratic form appearing
in \eqref{ass:expansion} and its relaxed version introduced in Definition~\ref{D:relaxation-formula}.

\begin{lemma}
  \label{lem:1}
  Let $W$ be as in Assumption~\ref{assumption} and let $Q$ be the quadratic form
  associated to $W$ through the expansion \eqref{ass:expansion}. Then
  \begin{enumerate}[(i)]
  \item[(Q1)] $Q(\cdot,G)$ is $Y$-periodic and
      measurable for all $G\in\R^{3\times 3}$,
  \item[(Q2)] for  almost every $y\in\R^2$ the map
      $Q(y,\cdot)$ is quadratic and satisfies
    \begin{equation*}
      c_1|\sym G|^2\leq Q(y,G)=Q(y,\sym G)\leq c_2|\sym G|^2\qquad\text{for all $ G\in\R^{3\times 3}$.}
    \end{equation*}
  \end{enumerate}
  Furthermore,  there exists a monotone function $r:\R_+\to\R_+\cup\{+\infty\}$ that can be chosen only depending on the
  parameters $c_1,c_2$ and $\rho$, such that $r(\delta)\to 0$ as
  $\delta\to 0$ and
  \begin{equation}\label{eq:94}
    \forall  G\in\R^{3\times 3}\,:\,|W(y,\id+ G)-Q(y, G)|\leq| G|^2r(| G|)
  \end{equation}
  for  almost every $y\in\R^2$.
\end{lemma}
(For a proof see \cite[Lemma~2.7]{Neukamm-12}.)

\begin{lemma}\label{L:Qgamma}
  \begin{enumerate}[(a)]
  \item Let $\gamma\in(0,\infty)$.  For all $ A \in \R^{2\times
      2}_{\sym}$  there exist a unique pair $(B,\phi)$ with  $B \in  \R^{2\times 2}_{\sym}$
    and $\phi \in \mathring H^1(I\times\mathcal Y,\R^3)$ such that:
  \begin{equation*}
    Q_{2,\gamma}(A)=\iint_{I\times
      Y}Q\left(y,\;\Lambda(x_3, A, B)+(\nabla_y
      \phi \ , \ \tfrac{1}{\gamma}\partial_3\phi)\,\right) dydx_3
  \end{equation*}
  The induced mapping $\R^{2\times 2}_{\sym}\ni A\mapsto(B,\phi)\in
  \R^{2\times 2}_{\sym}\times\mathring H^1(I\times\mathcal Y,\R^3)$ is
  bounded and linear.

\item Let $\gamma=\infty$.  For all $ A \in \R^{2\times 2}_{\sym}$
  there exist a unique triple $(B,\phi,d)$ with $B \in  \R^{2\times
    2}_{\sym}$, $\phi\in L^2(I,\mathring H^1(\mathcal Y,\R^3))$ and
  $d\in L^2(I,\R^3)$ such that:
  \begin{equation*}
    Q_{2,\infty}(A)=\iint_{I\times
      Y}Q\left(y,\;\Lambda(x_3, A, B)+(\nabla_y
    \phi \ , \ d)\,\right) dydx_3
  \end{equation*}
  The induced mapping $\R^{2\times 2}_{\sym}\ni A\mapsto(B,\phi,d)\in
  \R^{2\times 2}_{\sym}\times L^2(I,\mathring H^1(\mathcal
  Y,\R^3))\times L^2(I,\R^3)$ is
  bounded and linear.
  \end{enumerate}
 \end{lemma}

\begin{proof}[Proof of Lemma~\ref{L:Qgamma}]
  The proof is standard.

We only comment on the proof of part (a), i.~e. for $\gamma\in(0,\infty)$. We
  start with the following Korn-type inequality

  \begin{equation}\label{L:Qgamma:korn}
    \forall\psi\in \mathring
    H^1(I\times\mathcal Y,\R^3)\,:\quad
    \|\psi\|_{H^1(I\times Y,\R^3)}^2\leq C\iint_{I\times Y}|\sym (\nabla_y
    \psi \ , \ \tfrac{1}{\gamma}\partial_3\psi)|^2
  \end{equation}
  for some constant $C>0$. Since $Q$ is elliptic in the sense of
  Lemma~\ref{lem:1} (Q2), and since $Q$ vanishes for skew-symmetric matrices, for
  each pair $(A,B)$ of symmetric $2\times2$ matrices we can find
  a unique function $\phi=\phi_{A,B}\in  \mathring
  H^1(I\times\mathcal Y,\R^3)$ minimizing the integral
  \begin{equation}\label{eq:13}
    \t Q_\gamma(A,B):=\iint_{I\times Y}Q\left( y,\;\Lambda(x_3, A, B)+(\nabla_y
      \phi \ , \ \tfrac{1}{\gamma}\partial_3\phi)\,\right)\,d  y\,d x_3.
  \end{equation}
  Evidently $\phi$ depends linearly on
  $(A, B)$. In particular, for $i=1,\ldots,3$ there exist
  $E_i \in\mathring H^1(I\times \mathcal Y, \R^{2\times 2}_{\sym})$, $F_i \in\mathring H^1(I\times \mathcal Y, \R^{2\times 2}_{\sym})$ such
  that
  \begin{eqnarray} \label{eq:10001}
    & & \phi_i(x_3,y)=E_i(x_3,y):A+ F_i (x_3,y):B.
  \end{eqnarray}
  As a consequence $\t Q_\gamma$ is a quadratic form  and it is easy to
  check that there exist positive constants $c_{\gamma,1}<c_{\gamma,2}$ such that
  $$
c_{\gamma,1} (|A|^2+|B|^2 ) \leq \t Q_\gamma (A,B) \leq c_{\gamma,2} (|A|^2+|B|^2 ).
$$
  Hence, we conclude that there exists a bounded, positive definite
  operator \mbox{$\mathcal{A}: \R^{2\times 2}_{\sym}\times \R^{2\times
      2}_{\sym}\to \R^{2\times 2}_{\sym}\times \R^{2\times 2}_{\sym}$} such that
  $$\t Q_\gamma (A,B)= \langle\mathcal{A}(A,B),\,(\sym
  A,B)\rangle,$$
  where $\langle\cdot,\cdot\rangle$ denotes the inner product in $\R^{2\times 2}_{\sym}\times \R^{2\times
    2}_{\sym}$. By a block-wise decomposition of the right-hand side we get
  $$
  \t Q(A,B)=\mathcal{A}_1
  A : A+ \mathcal{A}_2 (A) : B
  +\mathcal{A}_3 B : B$$
  where $\mathcal{A}_1,\mathcal{A}_2,\mathcal{A}_3 \in L(\R^{2\times
    2}_{\sym}) $ are bounded operators on $\R^{2\times2}_{\sym}$ and
  $\mathcal A_1$ and $\mathcal A_3$ are positive definite. A
  straightforward calculation shows that this expression is minimized
  (with respect to  $B$) by $B_A=-\frac{1}{2} \mathcal{A}_3^{-1} \mathcal{A}_2
  \sym A$. Since this expression is obviously linear in $A$, the desired
  pair of functions associated with $A$ is given by $B_A$ and
  $\phi=\phi_{A,B_A}$.
\end{proof}

\begin{proof}[Proof of Theorem~\ref{T:main} (lower bound)]
  Without loss of generality we may assume that
  $\fint_{\Omega}u^h\,dx=0$ and $\limsup_{h\to
    0}h^{-2}\mathcal E^{h,\eh}(u^h)<\infty$. We only consider the case
  $\gamma\in(0,\infty)$. The proof in the case $\gamma=\infty$ holds verbatim. In view of
  \eqref{ass:non-degenerate}, the sequence $u^h$ has finite
  bending energy and the sequence $E^h$, see \eqref{eq:strain}, is bounded in
  $L^2(\Omega,\R^{3\times 3})$. Hence, from Theorem~\ref{T:FJM-compactness}
  we deduce that $u\in\WW(S,\R^3)$. By Lemma~\ref{kompts} (i) and Proposition~\ref{P:strain}
  (i) we can pass to a subsequence such that, for some $E\in L^2(\Omega\times Y; \R^{3\times 3})$,
  \begin{equation*}
    E^h\wtto E,
  \end{equation*}
  where $E$ can be written in the form of \eqref{P:strain:1a}. As in \cite{FJM1}
  a careful Taylor expansion of $W(\frac{x'}{\eh},\id+h E^h(x))$
  combined with the lower semi-continuity of convex integral functionals
  with respect to  weak two-scale convergence (see e.g. \cite[Proposition~1.3]{Visintin-07}) yields the lower bound
  \begin{equation*}
    \liminf\limits_{h\to 0}\frac{1}{h^2}\mathcal E^{h,\eh}(u^h)\geq
    \iint_{\Omega\times Y}Q(x',y,E(x,y))\,dy\,dx.
  \end{equation*}
  Combined with \eqref{P:strain:1a} the right-hand side becomes
  \begin{equation*}
    \iint_{Q\times Y}Q\Bigg(\,x',y,\Lambda(\secf(x'),B(x')+\left (\nabla_y \phi(x, y)\,,\,\tfrac{1}{\gamma}
  \partial_3 \phi(x, y) \right)\,\Bigg)\,dy\,dx,
  \end{equation*}
  where we used that $Q(x',y,F)$ only depends on $F$. Minimization
  over $B\in L^2(S,\R^{2\times 2})$ and $\phi\in L^2(S,
  H^1(I\times\mathcal Y,\R^3))$ yields
  \begin{equation*}
    \liminf\limits_{h\to 0}\frac{1}{h^2}\mathcal E^{h,\eh}(u^h)\geq
    \int_{S}Q_{2,\gamma}(\secf(x'))\,dx'=\mathcal E_\gamma(u).
  \end{equation*}
\end{proof}

It remains to prove the upper bound. As in \cite{Schmidt-07} and other related results,
the key ingredient here is the density result for $W^{2,2}$ isometric immersions
established in \cite{Hornung-arma1, Hornung-arma2} (cf. also \cite{Pakzad} for an earlier result in this direction).
It is the need for the results in \cite{Hornung-arma2} that forces us to consider domains $S$
which are not only Lipschitz but also piecewise $C^1$ -- more precisely, we only need that the outer unit normal
be continuous away from a subset of $\partial S$ with length zero.

For brevity, we denote by $\mathcal A(S)$ the set of all
$u\in W^{2,2}_{\delta}(S, \R^3)\cap C^{\infty}(\overline{S}, \R^3)$
with the following property:
\\
For  each $B\in C^{\infty}(\overline{S}, \R^{2\times 2}_{\sym})$ with satisfying
$B = 0$ in a neighborhood of $\{x\in S : \Pi(x) = 0 \}$, there exist $\a\in
C^{\infty}(\overline{S})$ and $g\in C^{\infty}(\overline{S}, \R^2)$ such that
\begin{equation}
\label{le71-001}
B = \sym\nabla' g + \a\Pi.
\end{equation}

The key ingredient in the proof of the upper bound is the following lemma.

\begin{lemma}
\label{le71}
The set $\mathcal A(S)$ is dense in $W^{2,2}_{\delta}(S)$ with respect to the strong $W^{2,2}$-topology.
\end{lemma}
\begin{proof}
This follows by combining the construction from \cite{Hornung-arma1, Hornung-arma2}
with the arguments in \cite{Schmidt-07} leading to his Lemma 3.3.
His result was recently re-derived in \cite{HoLePa} in a slightly different context.
\end{proof}

Thanks to Lemma \ref{le71} it will be enough to construct recovery sequences for limiting
deformations belonging to $\mathcal{A}(S)$. First we present a construction assuming the
existence of $\a$ and $g$ satisfying \eqref{le71-001}.

\begin{lemma}
  \label{le73}
  For $u\in\WW(S,\R^3)\cap W^{2,\infty}(S,\R^3)$, $\alpha\in
  W^{1,\infty}(S)$ and $g\in W^{1,\infty}(S,\R^2)$ define
  \begin{equation*}
    B(x')=\sym\nabla'g(x')+\alpha(x')\secf(x').
  \end{equation*}
  \begin{enumerate}[(a)]
  \item Let $\gamma\in(0,\infty)$ and
    $\phi\in C^\infty_c(S,C^\infty(I\times\mathcal
    Y,\R^3))$. Then there exists a sequence
    $\{u^h\}_{h>0}\subset H^1(\Omega, \R^3)$ such that
    \begin{equation*}
      u^h\to u,\qquad \nabla_h u^h\to
      (\nabla'u,\,n)\qquad\text{uniformly in }\Omega
    \end{equation*}
    and
    \begin{multline}
      \label{le73-1a}
      \lim_{h \to 0} \frac{1}{h^2}\mathcal E^{h,\eh}(u^h)\\
      =\iint_{\Omega\times Y}Q\left(
        y,\;\Lambda(\secf(x'),B(x'))+(\nabla_y\phi,\tfrac{1}{\gamma}\partial_3\phi)\,\right)\,d  y \,d x_3 \,d x'.
    \end{multline}
  \item Let $\gamma=\infty$, $\phi\in C^\infty_c(\Omega,C^\infty(\mathcal
    Y,\R^3))$ and $d\in C^\infty_c(\Omega,\R^3)$. Then there exists a sequence
    $\{u^h\}_{h>0}\subset H^1(\Omega, \R^3)$ such that
    \begin{equation*}
      u^h\to u,\qquad \nabla_h u^h\to
      (\nabla'u,\,n)\qquad\text{uniformly in }\Omega
    \end{equation*}
    and
    \begin{multline}
      \label{le73-1b}
      \lim_{h \to 0} \frac{1}{h^2}\mathcal E^{h,\eh}(u^h)\\
      =\iint_{\Omega\times Y}Q\left(
        y,\;\Lambda(\secf(x'),B(x'))+(\nabla_y\phi,\,d)\,\right)\,d  y \,d x_3 \,d x'.
    \end{multline}
  \end{enumerate}
\end{lemma}
\begin{proof}
  We start with the case $\gamma\in(0,\infty)$ and follow \cite[Theorem~{3.2} (ii)]{Schmidt-07}. Consider
  \begin{eqnarray*}
    v^h(x)&:=&u(x')+h\Big[(x_3+\alpha(x'))n(x')+\big(g(x')\cdot\nabla'\big)u(x')\Big],\\
    \quad R(x')&:=&(\nabla'u(x'),\ n(x')).
  \end{eqnarray*}
  Here $(g\cdot\nabla')u$ stands for
  $\sum_{\alpha=1,2}g_\alpha\partial_\alpha u$.
  A direct calculation (see \cite{Schmidt-07}) shows that
  \begin{equation*}
    R^t\nabla_h v^h=\id+h G
  \end{equation*}
  with
  \begin{equation*}
    G:=\left(\begin{array}{cc}
        x_3\secf(x')+ B &
        \begin{array}{c}0\\0\end{array}\\
        -b^t & 0\\
      \end{array}\right),\qquad
    B:=\nabla'g+\alpha\secf,\qquad b:=-\left(\begin{array}{c}
        \partial_1\alpha\\
        \partial_2\alpha
      \end{array}\right)+\secf g.
  \end{equation*}
  To $v^h$ we add an oscillating correction and a term compensating
  for $b$:
  \begin{equation}\label{eq:15}
    u^h(x):=v^h(x)+h\eh \t\phi(x,\frac{x'}{\eh})
  \end{equation}
  where
  \begin{equation*}
    \t\phi(x,y):=R(x')\Big[\,\phi(x,y)+\gamma x_3
    \left(\begin{array}{c}b\\0\end{array}\right)\,\Big].
  \end{equation*}
  Since $\alpha,g$ and $\phi$ are sufficiently smooth, the uniform
  convergence of $u^h$ and $\nabla_hu^h$ directly follows from the
  construction. Moreover, we have
  \begin{equation*}
    R^t\nabla_h u^h=\id+h(G+G^h)
  \end{equation*}
  where
  \begin{equation*}
    G^h:=\left(\,\nabla_y\phi(x,\tfrac{x'}{\eh}),\
      \tfrac{1}{\gamma}\partial_3\phi(x,\tfrac{x'}{\eh}) +
      \tfrac{\eh}{h}\gamma
      \left(\begin{array}{c}b\\0\end{array}\right)\,\right) + C^h.
  \end{equation*}
  Here the remainder term $C^h$ satisfies $\limsup_{h\to0}\frac{|C^h|}{\eh}<\infty$.
  Again, since $u$, $\phi$, $\alpha$ and $g$ are sufficiently smooth, we have
  \begin{equation}
    \limsup\limits_{h\to 0}h\sup_{x\in\Omega}\left(\,|G(x)|+|G^h(x)|\,\right)=0.
  \end{equation}
  Hence, \eqref{ass:frame-indifference}, \eqref{ass:expansion} and
  \eqref{eq:94} yield
  \begin{equation*}
    \limsup\limits_{h\to 0}\left|\frac{1}{h^2}\mathcal
    E^{h,\eh}(u^h)-\int_{\Omega}Q(\tfrac{x'}{\eh},G(x)+G^h(x))\,dx\right|=0,
  \end{equation*}
  and it suffices to show that
  \begin{equation}\label{eq:6}
    \lim\limits_{h\to 0}\int_{\Omega}Q(\tfrac{x'}{\eh},G(x)+G^h(x))\,dx=\text{[R.~H.~S. of \eqref{le73-1a}]}.
  \end{equation}
  By construction the sequence $G+G^h$ strongly two-scale converges to
  some limit $\t G\in L^2(\Omega\times Y,\R^{3\times 3})$ with
  \begin{equation*}
    \sym\t G=\Lambda(\secf,\sym B)
    +\sym\left(\,\nabla_y\phi(x,y),\
      \tfrac{1}{\gamma}\partial_3\phi(x,y)\,\right).
  \end{equation*}
  Hence, by the continuity of convex integral functionals with respect to
  strong two-scale convergence we can pass to the
  limit in \eqref{eq:6} (e.~g. see \cite[Lemma~4.8]{Neukamm-12}). Since $Q(y,F)$ only depends on the symmetric
  part of $F$, this completes the proof for the case
  $\gamma\in(0,\infty)$.

  The proof for $\gamma=\infty$ is similar to the above reasoning.
  Essentially, we only need to replace \eqref{eq:15} by
  \begin{equation*}
     u^h(x)=v^h(x)+h\eh \t \phi(x,\tfrac{x'}{\eh})+h^2 \t d(x),
  \end{equation*}
  where
  \begin{equation*}
     \t \phi(x,y)=R(x')\phi(x,y)\quad\text{and}\quad\t d(x)=R(x')\Big[\int_{-1/2}^{x_3} d(x',t) dt+ x_3\left(\begin{array}{c}b(x')\\0\end{array}\right)\,\Big].
  \end{equation*}
\end{proof}

\begin{proof}[Proof of Theorem~\ref{T:main} (Upper bound)]
  We only consider the case $\gamma\in(0,\infty)$ since the argument
  for $\gamma=\infty$ is the same. We may assume that $\mathcal E_\gamma(u)<\infty$, so
  $u\in\WW(S,\R^3)$. Moreover, since $Q_{2,\gamma}$ (see Lemma~\ref{L:Qgamma}) is continuous, it suffices to prove the
  statement for a dense subclass of $\WW(S,\R^3)$. Hence, by virtue of
  Lemma~\ref{le71}, we may assume without loss of generality that
  $u\in\mathcal A(S)$.
\\
By Lemma~\ref{L:Qgamma} there exist $B\in L^2(S,\R^{2\times 2})$ and
  $\phi\in L^2(S,H^1(I\times\mathcal Y,\R^3))$ such that
  \begin{equation}\label{eq:12}
    \mathcal E^\gamma(u)=\iint_{\Omega\times Y}Q(y,\Lambda(\secf,B)+\sym(\nabla_y\phi,\
    \tfrac{1}{\gamma}\partial_3\phi))\,dydx.
  \end{equation}
  Since $B(x')$ linearly depends on $\secf(x')$ we know in addition that
  $B(x')=0$ for $x'\in\{\,\secf=0\,\}$.

  By a density argument it suffices to show the following: There
  exists a doubly indexed sequence $u^{h,\delta}\in
  H^1(\Omega,\R^3)$ such that
  \begin{align}
    \label{eq:9}
    &\limsup\limits_{\delta\to 0}\limsup\limits_{h\to
      0}||u^{h,\delta}-u||_{L^2(\Omega,\R^3)}=0,\\
    \label{eq:11}
    &\limsup\limits_{h\to 0}\left|\frac{1}{h^2}\mathcal
      E^{h,\eh}(u^{h,\delta})-\mathcal E_\gamma(u)\right|<\delta.
  \end{align}
  Indeed, if this is the case then we obtain the desired sequence by
  extracting a diagonal sequence (e.~g. by appealing to
  \cite[Corollary 1.16]{Attouch-84}).

  We construct  $u^{h,\delta}$ as follows: By a density argument we
  may choose for each $\delta>0$ functions $B^\delta\in
  C^\infty(\overline S,\R^{2\times 2}_{\sym})$ and $\phi^\delta\in
  C^\infty_c(S,C^\infty(I\times\mathcal Y,\R^3))$ such that
  \begin{align}
    \label{eq:7}
    &||B^\delta-B||_{L^2(S)}+||\nabla_y\phi^\delta-\nabla_y\phi||_{L^2(\Omega\times
      Y)}+||\partial_3\phi^\delta-\partial_3\phi||_{L^2(\Omega\times
      Y)}\leq
    \delta^2,\\
    \label{eq:8}
    &B^\delta=0\text{ in a neighborhood of }\{\,\secf=0\,\}.
  \end{align}
  Because $u\in\mathcal
  A(S,\R^3)$ and due to \eqref{eq:8} we
  can find for each $\delta>0$ smooth functions $\alpha^\delta$ and
  $g^\delta$ such that
  $B^\delta=\sym\nabla'g^\delta+\alpha^\delta\secf$. We apply
  Lemma~\ref{le73} to $u$, $g^\delta,\alpha^\delta$ and $\phi^\delta$.
  We obtain a sequence $\{u^{h,\delta}\}_{h>0}$ that uniformly converges
  to $u$ as $h\to 0$. Hence, \eqref{eq:9} is satisfied. Moreover,
  Lemma~\ref{le73} yields
  \begin{equation*}
    \lim\limits_{h\to 0}\frac{1}{h^2}\mathcal
    E^h(u^{h,\delta})=\iint_{\Omega\times Y}Q(y,\Lambda(\secf,B^\delta)+\sym(\nabla_y\phi^\delta,\tfrac{1}{\gamma}\partial_3\phi^\delta)\,dydx.
  \end{equation*}
  By continuity of the functional on the right-hand side,
  combined with \eqref{eq:7} and \eqref{eq:12}, the bound \eqref{eq:11}
  follows.
\end{proof}

{\bf Acknowledgements.} Hornung and Vel\v{c}i\'{c} were supported by Deutsche Forschungsgemeinschaft grant no.
HO4697/1-1.
\\

Peter Hornung and Stefan Neukamm

{\sc Max-Planck-Insitut f\"ur Mathematik in den Naturwissenschaften
\\
Inselstr. 22
\\
04103 Leipzig
\\
Germany
}
\\

Igor Vel\v{c}i\'{c}

{\sc Max-Planck-Insitut f\"ur Mathematik in den Naturwissenschaften
\\
Inselstr. 22
\\
04103 Leipzig
\\
Germany
}

and

{\sc Basque Center of Applied Mathematics
\\
Alameda de Mazarredo 14
\\
48009 Bilbao
\\
Spain
}

\bibliographystyle{alpha}

\end{document}